\documentclass[reqno,12pt]{amsart}

\usepackage[utf8]{inputenc}
\usepackage{microtype}
\usepackage{amsfonts}
\usepackage{amsmath}
\usepackage{graphicx}
\usepackage{float}
\usepackage[dvipsnames]{xcolor}
\usepackage{hyperref}
\usepackage{tikz-cd}

\usepackage{amssymb}
\usepackage{enumitem}
\usepackage{mathrsfs}

\usepackage{tikz}
\usetikzlibrary{topaths}
\usetikzlibrary{calc}

\usepackage{a4wide}

\usepackage{empheq}

\newtheorem{theorem}{Theorem}[section]
\newtheorem*{theorem*}{Theorem}
\newtheorem{lemma}[theorem]{Lemma}
\newtheorem{proposition}[theorem]{Proposition}
\newtheorem*{proposition*}{Proposition}

\newtheorem*{corollary*}{Corollary}
\newtheorem{conjecture}[theorem]{Conjecture}
\newtheorem{cit}[theorem]{Citation}
\newtheorem*{conjecture*}{Conjecture}

\newtheorem*{question*}{Question}

\newtheorem*{main:non_inn_am}{Theorem~\ref{thrm:non_inn_am}}

\theoremstyle{definition}
\newtheorem{definition}[theorem]{Definition}

\newcommand{\N}{\mathbb{N}}
\newcommand{\Z}{\mathbb{Z}}

\newcommand{\R}{\mathbb{R}}

\newcommand{\Cantor}{\mathfrak{C}}
\newcommand{\vNA}{\mathcal{L}}

\DeclareMathOperator{\PSL}{PSL}

\numberwithin{equation}{section}

\begin{document}

\title[Non-inner amenability of the Higman--Thompson groups]{Non-inner amenability of the\\ Higman--Thompson groups}
\date{\today}
\subjclass[2020]{Primary 43A07;   
                 Secondary 20F65} 

\keywords{Higman--Thompson group, inner amenability, centralizers, Ping-Pong Lemma}

\author[E.~Bashwinger]{Eli Bashwinger}
\address{Department of Mathematics, University of Utah, Salt Lake City, UT 84112}
\email{u6060268@utah.edu}

\author[M.~C.~B.~Zaremsky]{Matthew C.~B.~Zaremsky}
\address{Department of Mathematics and Statistics, University at Albany (SUNY), Albany, NY 12222}
\email{mzaremsky@albany.edu}

\begin{abstract}
We prove that the Higman--Thompson groups $T_n$ and $V_n$ are non-inner amenable for all $n\ge 2$. This extends Haagerup and Olesen's result that Thompson's groups $T=T_2$ and $V=V_2$ are non-inner amenable. Their proof relied on machinery only available in the $n=2$ case, namely Thurston's piecewise-projective model for Thompson's group $T$, so our approach necessarily utilizes different tools. This also provides an alternate proof of Haagerup--Olesen's result when $n=2$.
\end{abstract}

\maketitle
\thispagestyle{empty}

\section*{Introduction}

Thompson's groups $F$, $T$, and $V$ were introduced by Thompson in the 1960s (see, e.g., \cite{mckenzie73}), and have since attained a prominent position in group theory. They often serve as examples of groups with unusual behavior and generate interesting open questions, most famously it remains open whether $F$ is amenable. These groups can be defined as certain groups of self-homeomorphisms of the standard binary Cantor set $\Cantor_2=\{1,2\}^\N$, and have natural generalizations $F_n$, $T_n$, and $V_n$ using the $n$-ary Cantor set $\Cantor_n=\{1,\dots,n\}^\N$. The groups $V_n$ were first studied by Higman \cite{higman74}, and are now typically called the Higman--Thompson groups. The groups $F_n$ and $T_n$ were technically not studied by Higman, and first appeared in work of Brown \cite{brown87}, but here we will refer to $F_n$, $T_n$, and $V_n$ collectively as Higman--Thompson groups. (As Brown says in \cite{brown87}, $F_n$ and $T_n$, ``\emph{were not considered by Higman, but they are simply the obvious generalizations of Thompson's $F$ and $T$.}'') As a remark, the groups we denote here by $F_n$, $T_n$, and $V_n$ were denoted $F_{n,\infty}$, $T_{n,1}$, and $G_{n,1}$, respectively, in \cite{brown87}.

A (discrete) group $G$ is \emph{inner amenable} if there is a finitely additive probability measure on $G\setminus\{1\}$ that is invariant under the conjugation action of $G$. We say $G$ is \emph{amenable} if there is a finitely additive probability measure on $G$ that is invariant under left multiplication. It is a standard fact that amenable groups are inner amenable, but the converse is far from true (for example see \cite{tucker-drob20}). If $G$ has a non-trivial finite conjugacy class then it is inner amenable for uninteresting reasons, so we typically either require the measure to be atomless (as in \cite{tucker-drob20}), or require $G$ to be ICC, meaning the conjugacy class of every non-trivial element is infinite. All the groups we consider here are ICC.

In \cite{haagerup17}, Haagerup and Olesen proved that $T$ and $V$ are not inner amenable. Our main result here is that the same is true of all the Higman--Thompson groups $T_n$ and $V_n$:

\begin{main:non_inn_am}
For all $n\ge 2$, the Higman--Thompson groups $T_n$ and $V_n$ are non-inner amenable.
\end{main:non_inn_am}

The proof in \cite{haagerup17} that $T$ and $V$ are non-inner amenable relies on the so called Thurston model for $T$, as the group of piecewise-$\PSL_2(\Z)$ homeomorphisms of the real projective line; see \cite[Section~3]{haagerup17} and \cite{fossas11} for more details. This does not generalize to $T_n$ for $n>2$, and so their proof cannot be adapted to $T_n$ and $V_n$. Our strategy here has essentially the same first step as theirs: it suffices to find a non-amenable subgroup $H$ of $T_n$ such that the centralizer in $H$ of any non-trivial $\alpha\in V_n$ is amenable (Citation~\ref{cit:subgroup_trick}). But without the Thurston model in the $n>2$ case, the rest of their proof has no analog. Instead, we construct such an $H$ dynamically, by applying a version of the Ping-Pong Lemma to carefully chosen elements of $T_n$ (Proposition~\ref{prop:one_fixed}), constructed in such a way to ensure the restrictions to certain regions of the circle are contracting. One crucial tool we use to show that this $H$ works is the description of the centralizer $C_{V_n}(\alpha)$ from \cite{bleak13}. We should mention that when $n=2$ our approach also provides a new proof of Haagerup--Olesen's result, which is arguably an improvement since it is not as beholden to the $T$-specific Thurston model.

This completes the picture for $T_n$ and $V_n$, and we should remark that there is also a complete picture for $F_n$. In \cite{jolissaint98}, Jolissaint proved that Thompson's group $F$ is inner amenable, and Picioroaga proved in \cite{picioroaga06} that all the $F_n$ are inner amenable. In fact the group von Neumann algebras $\vNA(F_n)$ of the $F_n$ are all McDuff factors, and so have Property $\Gamma$, which are stronger properties. See \cite{bashwinger23} for additional examples of generalizations of $F$ yielding McDuff factors, such as the pure braided Thompson group $bF$. As a remark, the fact that none of the $T_n$ or $V_n$ are inner amenable implies that none of their group von Neumann algebras $\vNA(T_n)$, $\vNA(V_n)$ are McDuff factors, nor do they have Property $\Gamma$. We should also point out that amenable groups have all of these properties.

To summarize, we have the following chain of properties for a (discrete, ICC) group $G$:
\[
G \text{ amenable} \Rightarrow \vNA(G) \text{ McDuff} \Rightarrow \vNA(G) \text{ has Property }\Gamma \Rightarrow G \text{ inner amenable,}
\]
we now know that the groups $T_n$ and $V_n$ satisfy none of these, and the groups $F_n$ satisfy all of them except amenability is open. As a remark, it would interesting to try and understand whether the groups $T_n$ and $V_n$ are all properly proximal, in the sense of Boutonnet--Ioana--Peterson \cite{boutonnet21}. This is a property that inner amenable groups cannot have (even up to measure equivalence), and it is an open problem whether there can exist a non-properly proximal group that is not measure equivalent to any inner amenable group; see \cite[Question~1(c)]{boutonnet21} and \cite{ishan24}. Also see \cite{ding23} for a wealth of additional results on proper proximality.

We should mention that the ``Higman--Thompson groups'' often include further generalizations denoted $V_{n,r}$ (and $T_{n,r}$), but when $r>1$ we do not have a description of the centralizers like in \cite{bleak13} for the $r=1$ case. If it turns out that centralizers in $V_{n,r}$ behave similarly to centralizers in $V_n=V_{n,1}$, then it seems likely our approach would show that all the $V_{n,r}$ and $T_{n,r}$ are non-inner amenable. Other related groups for which inner amenability is open in general include braided Thompson groups and R\"over--Nekrashevych groups, and it would be interesting to try and use similar tools to prove for example that none of these are inner amenable.

This paper is organized as follows. In Section~\ref{sec:thompson} we recall the construction of $T_n$ and $V_n$ as groups of self-homeomorphisms of $\Cantor_n$. In Section~\ref{sec:circle} we view $T_n$ as a group of self-homeomorphisms of the circle, and construct a particular free subgroup $H$ of $T_n$. In Section~\ref{sec:main_proof} we use the computation of centralizers in $V_n$ from \cite{bleak13}, and the particular free subgroup $H$, to prove our main result.

\subsection*{Acknowledgments} We are grateful to Matt Brin for helpful conversations, and to Jesse Peterson for help with tracking down references. The second author is supported by grant \#635763 from the Simons Foundation.

\section{Higman--Thompson groups}\label{sec:thompson}

In this section we recall the definitions of the Higman--Thompson groups $F_n$, $T_n$, and $V_n$. Let $\Cantor_n$ denote the $n$-ary Cantor set, that is, $\Cantor_n=\{1,\dots,n\}^\N$ with the usual product topology. The elements of $\Cantor_n$ are infinite words $\kappa$ in the alphabet $\{1,\dots,n\}$. Let $\{1,\dots,n\}^*$ denote the set of all finite words $w$ in the alphabet $\{1,\dots,n\}$. For each $w\in\{1,\dots,n\}^*$ let $\Cantor_n(w)$ denote the \emph{cone} $\Cantor_n(w)=\{w\kappa\mid \kappa\in \Cantor_n\}$. Note that $\Cantor_n(w)$ is canonically homeomorphic to $\Cantor_n$ via the map $w\kappa\mapsto \kappa$.

\begin{definition}[Higman--Thompson groups]\label{def:hig_thomp}
The \emph{Higman--Thompson group} $V_n$ is the group of self-homeomorphisms of $\Cantor_n$ of the following form. Take a partition of $\Cantor_n$ into finitely many cones $\Cantor_n(w_1),\dots,\Cantor_n(w_k)$, another partition into the same number of cones $\Cantor_n(v_1),\dots,\Cantor_n(v_k)$, and map each $\Cantor_n(w_i)$ to some $\Cantor_n(v_j)$ via the ``prefix replacement'' map $w_i\kappa\mapsto v_j\kappa$. This defines a self-homeomorphism of $\Cantor_n$, and $V_n$ is the collection of all such self-homeomorphisms. The groups $F_n$ and $T_n$ are the subgroups of $V_n$ defined as follows. Note that the lexicographic order on $\{1,\dots,n\}^*$ is a total order. When writing a partition $\Cantor_n(w_1),\dots,\Cantor_n(w_k)$, assume that $w_1<\cdots<w_k$. Now $F_n$ is the subgroup of $V_n$ consisting of all homeomorphisms as constructed above, such that for each $i$, $\Cantor_n(w_i)$ maps to $\Cantor_n(v_i)$, and $T_n$ is the subgroup of all homeomorphisms such that for each $i$, if $\Cantor_n(w_i)$ maps to $\Cantor_n(v_j)$ then $\Cantor_n(w_{i+1})$ maps to $\Cantor_n(v_{j+1})$ (with subscripts taken modulo $k$).
\end{definition}

As a remark, it is not immediately obvious from this definition that $F_n$, $T_n$, and $V_n$ are groups, i.e., that they are closed under composition, but it is a straightforward exercise to verify.

One crucial fact that we will need is the following:

\begin{lemma}\label{lem:free_subgroup}
Let $H\le T_n$ be a non-abelian free subgroup. Let $\kappa\in \Cantor_n$ be any point in the $T_n$-orbit of $111\cdots$. Then the stabilizer in $H$ of $\kappa$ is cyclic.
\end{lemma}

\begin{proof}
Since the properties of a subgroup being non-abelian, free, or cyclic are all preserved under conjugation, without loss of generality we can assume $\kappa=111\cdots$. The stabilizer in $T_n$ of $111\cdots$ is $F_n$, so the stabilizer in $H$ of $111\cdots$ is $H\cap F_n$. This is free, since it is a subgroup of the free group $H$, and $F_n$ contains no non-abelian free subgroups \cite[Theorem~3.1]{brin85}, so $H\cap F_n$ is abelian, hence cyclic.
\end{proof}

\section{Free groups acting on the circle}\label{sec:circle}

In addition to viewing $T_n$ as a group of self-homeomorphisms of $\Cantor_n$, we can also view it as a group of self-homeomorphisms of the circle $S^1$. More precisely, viewing the circle as $\R/\Z$, $T_n$ is the group of all orientation-preserving self-homeomorphisms of $S^1$ that are piecewise linear, with slopes powers of $n$ and breakpoints in $\Z[1/n]/\Z$ (this is the viewpoint taken for example in \cite{stein92}). With this viewpoint, it is easy to see that $T_n$ contains non-abelian free subgroups, using the Ping-Pong Lemma on $S^1$. For example, see \cite[Lemma~3.2]{bleak11} and its proof for an explicit recipe for building non-abelian free subgroups. In this section we will construct a specific non-abelian free subgroup $H$ in $T_n$ with certain features that will prove useful later.

Let us recall some basic dynamics. A continuous map $f\colon X\to X$ from a metric space $(X,d)$ to itself is called \emph{contracting} if $d(f(x),f(x'))<d(x,x')$ for all $x\ne x'$ in $X$. Given a homeomorphism $f\colon X\to X$ with a fixed point $x$, call $x$ an \emph{attracting} fixed point for $f$ if there is an open neighborhood $U$ of $x$ in $X$ such that for any $x'\in U$, the sequence $x',f(x'),f^2(x'),\dots$ converges to $x$. Call $x$ a \emph{repelling} fixed point for $f$ if it is an attracting fixed point for $f^{-1}$. Note that any continuous self-map $f$ of a closed interval has a fixed point (by the intermediate value theorem), and if $f$ is contracting then this fixed point is attracting and unique (since the distance between two different fixed points would not be able to contract). In what follows we will view $S^1$ as a metric space with the usual arclength metric, so in particular every non-empty proper connected closed subset of $S^1$ is isometric to a closed interval.

\begin{lemma}[Contracting Ping-Pong]\label{lem:pong}
Let $a$ and $b$ be a pair of orientation-preserving self-homeomorphisms of $S^1$, such that each $c\in\{a,b,a^{-1},b^{-1}\}$ has exactly one attracting fixed point in $S^1$. Suppose for each such $c$ that we have a connected closed set $P(c)\subseteq S^1$ containing the attracting fixed point of $c$, such that the $P(c)$ are pairwise disjoint. Assume for each such $c$ that $c(S^1\setminus P(c^{-1}))\subseteq P(c)$, and that the restriction of $c$ to $S^1\setminus P(c^{-1})$ is contracting. Then the group $\langle a,b\rangle$ is free of rank $2$, and every non-trivial element of $\langle a,b\rangle$ has exactly one attracting fixed point.
\end{lemma}

\begin{proof}
The usual Ping-Pong Lemma shows that $\langle a,b\rangle$ is free of rank $2$; let us briefly recall how this works for the sake of completeness. For any non-empty reduced word $g$ in the alphabet $\{a,b,a^{-1},b^{-1}\}$, say with rightmost letter $c$ and leftmost letter $d$, choose $x\in\{a,b,a^{-1},b^{-1}\}\setminus\{c^{-1},d\}$. Now the assumptions ensure that $g(P(x))\subseteq P(d)$ and $P(d)\cap P(x)=\emptyset$, so $g$ does not act as the identity on $P(x)$, which shows $g\ne 1$. Hence no non-empty reduced word represents the identity in $\langle a,b\rangle$, which means $\langle a,b\rangle$ is free of rank $2$.

Now let $g\in \langle a,b\rangle\setminus\{1\}$, and we need to show $g$ has exactly one attracting fixed point. Since the property of having exactly one attracting fixed point is preserved under conjugation, without loss of generality $g$ is cyclically reduced. Also, we may assume $g$ has word length greater than $1$, since we already know $a$, $a^{-1}$, $b$, and $b^{-1}$ each have exactly one attracting fixed point. Say the cyclically reduced word $g$ is of the form $g=dwc$ for some (possibly empty) $w$, i.e, the rightmost letter of $g$ is $c$ and the leftmost is $d$, for $d\ne c^{-1}$. We first claim that if $g(x)=x$, then $x\in P(c^{-1})$ or $x\in P(d)$. Indeed, if $x\not\in P(c^{-1})$ then $c(x)\in P(c)$, and the standard ping-pong argument shows that $x=g(x)=dwc(x)\in P(d)$. Next we claim that $g$ really does have a fixed point in each of $P(d)$ and $P(c^{-1})$. Indeed, since $d\ne c^{-1}$ we have $g(P(d))\subseteq P(d)$, and $P(d)$ is isometric to a closed interval (being a closed, connected, proper subset of $S^1$), so $g$ has a fixed point in $P(d)$. An analogous argument applied to $g^{-1}$ shows that there is a fixed point in $P(c^{-1})$. Now we claim the fixed point of $g$ in $P(d)$ is an attracting fixed point, and that it is the unique attracting fixed point for $g$. First note that, viewing $g$ as a map from $P(d)$ to itself, our assumptions about certain restrictions being contracting ensure that the fixed point is unique and attracting in $P(d)$. Also, note that since $S^1\setminus P(c^{-1})$ is open, $g(S^1\setminus P(c^{-1}))$ and hence $g(P(d))$ lies in the interior of $P(d)$, which implies that the attracting fixed point is even attracting in all of $S^1$. An analogous argument applied to $g^{-1}$ shows that there is a unique fixed point in $P(c^{-1})$ and it is an attracting fixed point for $g^{-1}$ in $S^1$. In particular this is a repelling fixed point for $g$, and so we conclude that the attracting fixed point in $P(d)$ is the unique attracting fixed point for $g$ in $S^1$.
\end{proof}

One could argue theoretically that there exist $a$ and $b$ in $T_n$ satisfying the conditions of Lemma~\ref{lem:pong}, by appealing to the fact that $\Z[1/n]$ is dense in $\R$, and taking large enough powers to ensure that the desired restrictions are contracting. However, for the sake of concreteness, in the next proof we will construct explicit elements $a$ and $b$ in $T_n$ satisfying these conditions.

\begin{proposition}\label{prop:one_fixed}
There exists a non-abelian free subgroup $H$ of $T_n$ such that every non-trivial element of $H$ has exactly one attracting fixed point in $S^1$.
\end{proposition}

\begin{proof}
We need to find elements $a,b\in T_n$ and sets $P(a),P(a^{-1}),P(b),P(b^{-1}) \subseteq S^1$ satisfying the assumptions of Lemma~\ref{lem:pong}. First let $f$ be the map $f\colon [0,1/n]\to[0,1/n]$ given by the piecewise definition
\[
f(x) = \left\{\begin{array}{ll}
    n^3 x &\text{if } 0\le x\le (n^3-1)/n^7 \\
    x + (n^6-2n^3+1)/n^7 &\text{if } (n^3-1)/n^7 \le x\le 1/n^4 \\
    x/n^3 + (n^3-1)/n^4 &\text{if } 1/n^4 \le x\le 1/n \text{.}
\end{array}\right.
\]
See Figure~\ref{fig:ponging} for a picture of the map $f$, with the domain copy of $[1/n]$ on the bottom and the range copy on the top. The two most important features are that both interior breakpoints in the range are to the right of both interior breakpoints in the domain, and $f$ restricts to a contracting map $[1/n^4,1/n]\to[1/n^4,1/n]$ (both of these features rely on the fact that $n\ge 2$).

\begin{figure}[htb]
\centering
\begin{tikzpicture}[line width=0.8pt]
  \filldraw[lightgray!50] (0,0) -- (10,0) -- (10,4) -- (0,4) -- (0,0);
  \draw[line width=1.2pt] (0,0) -- (10,0) -- (10,4) -- (0,4) -- (0,0);
  \draw (2,0) -- (6,4)   (4,0) -- (8,4);
  \draw[->] (-1,0) -- (-1,4);
  \node at (-1.5,2) {$f$};
  
  \node at (0,-.5) {$0$}; \node at (2,-.5) {$(n^3-1)/n^7$}; \node at (4,-.5) {$1/n^4$}; \node at (10,-.5) {$1/n$};
  
  \node at (0,4.5) {$0$}; \node at (5,4.5) {$(n^3-1)/n^4$}; \node at (7.9,4.5) {$(n^6-n^3+1)/n^7$}; \node at (10,4.5) {$1/n$};
  
  \node at (2,2) {slope $n^3$}; \node at (5,2) {slope $1$}; \node at (8,2) {slope $1/n^3$};
\end{tikzpicture}
\caption{The map $f\colon [0,1/n]\to[0,1/n]$ (not to scale).}
\label{fig:ponging}
\end{figure}

Let $h\colon [0,1/n]\to [1/n,1]$ be the linear map $h(x)=1 - (n-1)x$. We now define $a\in T_n$ by specifying $a(x+\Z)$ for each $x\in [0,1]$ as follows:
\[
a(x+\Z) = \left\{\begin{array}{ll}
    f(x)+\Z &\text{if } 0\le x\le 1/n \\
    h\circ f\circ h^{-1}(x) + \Z &\text{if } 1/n \le x\le 1 \text{.}
\end{array}\right.
\]
By construction the breakpoints of $a$ all lie in $\Z[1/n]/\Z$, and it is easy to check that all slopes are powers of $n$, so $a$ really is an element of $T_n$. Note that $a$ has a repelling fixed point at $0+\Z$, an attracting fixed point at $1/n + \Z$, and no other fixed points. Now let $b$ be the conjugate of $a$ by the rotation $x+\Z \mapsto x+1/n^2 + \Z$. Thus $b$ has a repelling fixed point at $1/n^2 + \Z$, an attracting fixed point at $(n+1)/n^2 + \Z$, and no other fixed points.

Now we will build the sets $P(a),P(a^{-1}),P(b),P(b^{-1})$. Intuitively, each will be an arc of length $1/n^3$ (using the notion of ``length'' in $S^1=\R/\Z$ coming from $\R$). We construct them as follows, roughly working clockwise from $0+\Z$:
\[
P(a^{-1}) \coloneqq \{x+\Z\mid x\in [-(n-1)/n^4,1/n^4]\}\text{,}
\]
\[
P(b^{-1}) \coloneqq \{x+\Z\mid x\in [(n^2-n+1)/n^4,(n^2+1)/n^4]\}\text{,}
\]
\[
P(a) \coloneqq \{x+\Z\mid x\in [(n^3-1)/n^4,(n^3+n-1))/n^4]\}\text{,}
\]
\[
P(b) \coloneqq \{x+\Z\mid x\in [(n^3+n^2-1)/n^4,(n^3+n^2+n-1)/n^4]\}\text{.}
\]
Note that for each $c\in\{a,a^{-1},b,b^{-1}\}$, $P(c)$ contains the attracting fixed point of $c$, namely
\[
0+\Z\in P(a^{-1}) \text{, } 1/n^2 + \Z \in P(b^{-1}) \text{, } 1/n + \Z \in P(a) \text{, and } (n+1)/n^2 + \Z \in P(b) \text{.}
\]
Also note that these four sets are pairwise disjoint, since the fact that $n\ge 2$ ensures that we have the following four inequalities:
\[
1/n^4 < (n^2-n+1)/n^4 \text{, }
\]
\[
(n^2+1)/n^4 < (n^3-1)/n^4 \text{, }
\]
\[
(n^3+n-1))/n^4 < (n^3+n^2-1))/n^4 \text{, and }
\]
\[
(n^3+n^2+n-1)/n^4 < 1-(n-1)/n^4 \text{.}
\]
See Figure~\ref{fig:circle} for a visualization of all of this.

\begin{figure}[htb]
 \begin{tikzpicture}[line width=0.8pt]
  \draw[pink, line width=10pt] (210:4) arc (210:150:4);
  \draw[Apricot, line width=10pt] (-30:4) arc (-30:30:4);
  \draw[SkyBlue, line width=10pt] (60:4) arc (60:120:4);
  \draw[SpringGreen, line width=10pt] (240:4) arc (240:300:4);
  \draw (0:0) circle (4cm);
  \filldraw (0:4) circle (2pt); \node at (0:5.2) {$(n+1)/n^2$};
  \filldraw (30:4) circle (2pt); \node at (30:5.5) {$(n^3+n^2-1)/n^4$};
  \filldraw (60:4) circle (2pt); \node at (60:5) {$(n^3+n-1)/n^4$};
  \filldraw (90:4) circle (2pt); \node at (90:4.5) {$1/n$};
  \filldraw (120:4) circle (2pt); \node at (120:5) {$(n^3-1)/n^4$};
  \filldraw (150:4) circle (2pt); \node at (150:5) {$(n^2+1)/n^4$};
  \filldraw (180:4) circle (2pt); \node at (180:4.7) {$1/n^2$};
  \filldraw (210:4) circle (2pt); \node at (210:5.5) {$(n^2-n+1)/n^4$};
  \filldraw (240:4) circle (2pt); \node at (240:4.7) {$1/n^4$};
  \filldraw (270:4) circle (2pt); \node at (270:4.5) {$0$};
  \filldraw (300:4) circle (2pt); \node at (300:5) {$(n^4-n+1)/n^4$};
  \filldraw (330:4) circle (2pt); \node at (330:5.7) {$(n^3+n^2+n-1)/n^4$};
  
  \draw[SkyBlue,->,line width=1.5pt] (240:3.8) to[out=110,in=240] (110:3.7);
  \draw[SkyBlue,->,line width=1.5pt] (300:3.8) to[out=70,in=300] (70:3.7);
	
 \end{tikzpicture}
 \caption{A visualization of $S^1$ with the ``ping-pong sets'' $P(a^{-1})$, $P(b^{-1})$, $P(a)$, and $P(b)$ color-coded, in green, pink, blue, and orange respectively. To label the points of interest (fixed points of $a$ and $b$ and endpoints of the ping-pong sets), we write $x$ rather than $x+\Z$. The blue arrows indicate that $a$ maps the endpoints of $P(a^{-1})$ (and consequently all of $S^1\setminus P(a^{-1})$) into the interior or $P(a)$. To avoid making the picture too messy, we have not labeled all points of interest, e.g., breakpoints where the slope changes, and so forth. Also, we remark that no attempt was made to draw the picture ``to scale'' in any sense.}\label{fig:circle}
\end{figure}

We need to show that $a(S^1\setminus P(a^{-1}))\subseteq P(a)$, with similar statements for the other three elements. If we identify $(0,1)$ homeomorphically with $S^1\setminus\{0+\Z\}$ via $x\mapsto x+\Z$, then $S^1\setminus P(a^{-1})$ is identified with the open interval $(1/n^4,1-(n-1)/n^4)$. The image of this under $a$ equals the open interval $((n^6-n^3+1)/n^7,1-(n-1)(n^6-n^3+1)/n^7)$. The set $P(a)$ is identified with $[(n^3-1)/n^4,(n^3+n-1))/n^4]$, so now we just need to check that
\[
(n^3-1)/n^4 \le (n^6-n^3+1)/n^7 \text{ and }
\]
\[
1-(n-1)(n^6-n^3+1)/n^7 \le (n^3+n-1))/n^4 \text{.}
\]
Indeed, these are readily verified. It is similarly straightforward to see that $a^{-1}(S^1\setminus P(a))\subseteq P(a^{-1})$, and since the situation for $b$ and $b^{-1}$ looks exactly the same as that for $a$ and $a^{-1}$, just rotated by $1/n^2 + \Z$, the analogous result holds for them as well.

Lastly, we need to show that the restriction of $a$ to $S^1\setminus P(a^{-1})$ is contracting, and so forth for the other elements. Identify $(0,1)$ with $S^1\setminus\{0+\Z\}$ as before, so $S^1\setminus P(a^{-1})$ is the open interval $(1/n^4,1-(n-1)/n^4)$. In the construction of $a$, we see that $a$ has slope $1/n^3$ on this entire open interval (to see this, it helps to note that $h(1/n^4)=1-(n-1)/n^4$). Thus, the restriction of $a$ to this interval is indeed contracting. Similarly, the restriction of $a^{-1}$ to $S^1\setminus P(a)$ coincides with a linear map with slope $1/n^3$, hence is contracting. Since the situation for $b$ and $b^{-1}$ is the same as that for $a$ and $a^{-1}$, just rotated by $1/n^2 + \Z$, a parallel argument handles them as well.
\end{proof}

As a remark, the virtually free group $\Lambda$ constructed in \cite{haagerup17} inside $T=T_2$ does not have a finite index free subgroup of the above form, since their $\Lambda$ contains elements with a single fixed point in $S^1$ (arising from parabolic elements of $\PSL_2(\Z)$). In fact, this $\Lambda$ is simply the subgroup of $T$ generated by one standard generator and the square of another, so our approach here is not only quite different, but also quite a bit more difficult, than the one in \cite{haagerup17}.

\section{Proof of non-inner amenability}\label{sec:main_proof}

In this section we prove our main result:

\begin{theorem}\label{thrm:non_inn_am}
For all $n\ge 2$, the Higman--Thompson groups $T_n$ and $V_n$ are non-inner amenable.
\end{theorem}

Just like in \cite{haagerup17}, we will use the following useful sufficient condition for non-inner amenability:

\begin{cit}\cite[Corollary~4.3]{haagerup17}\label{cit:subgroup_trick}
Let $H\le G$ be (discrete) groups. Suppose that $H$ is non-amenable, and that for all $g\in G\setminus\{1\}$ the centralizer $C_H(g)$ is amenable. Then $G$, and indeed any subgroup of $G$ containing $H$, is non-inner amenable.
\end{cit}

\begin{proof}[Proof of Theorem~\ref{thrm:non_inn_am}]
Let $H$ be any non-abelian free subgroup of $T_n$. (Toward the end of the proof we will specialize to $H$ satisfying the conditions in Proposition~\ref{prop:one_fixed}, but for most of this proof $H$ can be any non-abelian free subgroup of $T_n$.) We will apply Citation~\ref{cit:subgroup_trick}, using $V_n$ as $G$. This means that we want to show that for any $\alpha\in V_n\setminus\{1\}$ the centralizer $C_H(\alpha)$ is amenable (equivalently cyclic, since $H$ is free), and then we will conclude that any subgroup of $V_n$ containing $H$ is non-inner amenable, in particular this applies to $T_n$ and $V_n$.

The centralizer $C_{V_n}(\alpha)$ was computed in \cite{bleak13}. It decomposes as an (internal) direct product, which using roughly the notation of \cite[Corollary~5.2]{bleak13} is
\[
C_{V_n}(\alpha) = C_{V_n(\mathfrak{T}_\alpha)}(\alpha) \times C_{V_n(\mathfrak{Z}_\alpha)}(\alpha) \text{.}
\]
Let us explain these two factors. First, for $G\le V_n$ and $\mathfrak{X}\subseteq \Cantor_n$, let $G(\mathfrak{X})$ be the subgroup of $G$ consisting of all elements fixing every point in $\Cantor_n\setminus \mathfrak{X}$, i.e., all elements supported on $\mathfrak{X}$. The sets $\mathfrak{T}_\alpha$ and $\mathfrak{Z}_\alpha$ are certain open subsets of $\Cantor_n$ partitioning $\Cantor_n$. For the moment we do not have to worry about the precise definitions of $\mathfrak{T}_\alpha$ and $\mathfrak{Z}_\alpha$ (intuitively, $\mathfrak{T}_\alpha$ supports ``torsion behavior'' of $\alpha$ and $\mathfrak{Z}_\alpha$ supports ``non-torsion behavior''), except to first point out that $\mathfrak{Z}_\alpha$ is empty if and only if $\alpha$ has finite order. This follows from the definition of $\mathfrak{Z}_\alpha$ in \cite{bleak13} together with \cite[Lemma~4.4]{bleak13} (see also \cite[Proposition~6.1]{burillo09}).

\textbf{Infinite order:} Let us first focus on the case when $\alpha$ has infinite order, so $\mathfrak{Z}_\alpha$ is non-empty. Consider $C_H(\alpha)$, which we want to show is amenable. This is a free group, being a subgroup of the free group $H$. Its intersection with $C_{V_n(\mathfrak{T}_\alpha)}(\alpha)$ equals $C_{H(\mathfrak{T}_\alpha)}(\alpha)$, which is the group of elements of $H$ fixing all points in $\Cantor_n\setminus \mathfrak{T}_\alpha = \mathfrak{Z}_\alpha$. Since $\mathfrak{Z}_\alpha$ is open and non-empty, it contains a point in the $T_n$-orbit of $111\cdots$, this orbit being dense in $\Cantor_n$. Thus, $C_{H(\mathfrak{T}_\alpha)}(\alpha)$ is contained in the stabilizer of a point in this orbit, and so is cyclic by Lemma~\ref{lem:free_subgroup}. Since $C_{V_n(\mathfrak{T}_\alpha)}(\alpha)$ is normal in $C_{V_n}(\alpha)$, its intersection with $H$ is therefore a cyclic normal subgroup of $H$. Since $H$ is non-abelian and free, this means $C_{H(\mathfrak{T}_\alpha)}(\alpha)$ is trivial, since non-abelian free groups cannot have non-trivial cyclic normal subgroups. In particular, the projection map $C_{V_n}(\alpha)\to C_{V_n(\mathfrak{Z}_\alpha)}(\alpha)$ restricted to $C_H(\alpha)$ is injective.

Having shown that the free group $C_H(\alpha)$ embeds into $C_{V_n(\mathfrak{Z}_\alpha)}(\alpha)$, it now suffices to show that $C_{V_n(\mathfrak{Z}_\alpha)}(\alpha)$ is amenable, since this will imply $C_H(\alpha)$ is amenable as desired. By the formal statement of Theorem~1.1 of \cite{bleak13}, we have
\[
C_{V_n(\mathfrak{Z}_\alpha)}(\alpha) \cong \prod_{j=1}^t ((A_j\rtimes\Z)\times P_{q_j})
\]
for some $t\ge 0$ and some finite groups $A_j$ and $P_{q_j}$. We do not need to get into the precise details of what $A_j$ and $P_{q_j}$ are; since they are all finite, the factors $(A_j\rtimes\Z)\times P_{q_j}$ are all (elementary) amenable, and so $C_{V_n(\mathfrak{Z}_\alpha)}(\alpha)$ is also (elementary) amenable, and we are done.

\textbf{Finite order, in $V_n\setminus T_n$:} It remains to prove the result when $\alpha$ has finite order. In this case $\mathfrak{Z}_\alpha=\emptyset$, so the above approach does not work. First suppose $\alpha\in V_n\setminus T_n$. View $V_n$ as acting by left-continuous self-bijections of the circle (analogously to the corresponding situation in the proof of \cite[Theorem~4.4]{haagerup17} for $V_2=V$), so $\alpha$ has a positive finite number of points of discontinuity, which are all $n$-ary points. Since any element of $T_n$ is continuous, every element of $C_H(\alpha)$ must stabilize the set of points of discontinuity of $\alpha$, so any two elements of $C_H(\alpha)$ have non-trivial powers fixing some $n$-ary point. Since the stabilizer in $T_n$ of any $n$-ary point is isomorphic to $F_n$, which contains no non-abelian free subgroups, this shows that any two elements of $C_H(\alpha)$ have non-trivial powers that commute. Since $H$ is free, this implies $C_H(\alpha)$ is cyclic, as desired.

\textbf{Finite order, in $T_n$:} Finally, assume the finite order element $\alpha\ne 1$ lies in $T_n$. We will now finally specialize $H$ to no longer be an arbitrary non-abelian free subgroup but one satisfying the conditions in Proposition~\ref{prop:one_fixed}. In particular, as a group of self-homeomorphisms of $S^1$, every non-trivial element of $H$ has exactly one attracting fixed point. Since any element commuting with such an element must fix its unique attracting fixed point, this shows that if $C_H(\alpha)$ is non-trivial, $\alpha$ must fix a point of $S^1$. But the only finite order element of $T_n$ (or indeed of any group of orientation-preserving homeomorphisms of the circle) that fixes a point in $S^1$ is the identity, so we conclude that $C_H(\alpha)=\{1\}$.

We have shown that in all cases, $C_H(\alpha)$ is amenable, and so we are done.
\end{proof}

It would be interesting to extend our main result to R\"over--Nekrashevych groups. The \emph{R\"over--Nekrashevych group} $V_n(G)$ of a self-similar group $G$ of automorphisms of the infinite rooted $n$-ary tree is a group obtained by combining $V_n$ with $G$ in a certain way. See \cite{roever99,nekrashevych04,skipper21} for more details. We suspect that the above approach should reveal that all the $V_n(G)$ are non-inner amenable, but first one would need to understand centralizers in $V_n(G)$, and this seems quite difficult. Thus, we leave this as a conjecture:

\begin{conjecture}
The R\"over--Nekrashevych groups $V_n(G)$ are all non-inner amenable.
\end{conjecture}

\bibliographystyle{alpha}\newcommand{\etalchar}[1]{$^{#1}$}

\end{document}